\documentclass[lettersize,journal]{IEEEtran}
\usepackage{amsmath,amsfonts}
\usepackage{algorithmic}
\usepackage{algorithm}
\usepackage{array}
\usepackage[caption=false,font=normalsize,labelfont=sf,textfont=sf]{subfig}
\usepackage{textcomp}
\usepackage{stfloats}
\usepackage{url}
\usepackage{verbatim}
\usepackage{graphicx}
\usepackage{cite}
\usepackage{amsthm}
\usepackage{dirtytalk}
\newtheorem{thm}{Theorem}
\begin{document}

\title{Adaptive Pricing in Unit Commitment Under Load and Capacity Uncertainty}

\author{Dimitris Bertsimas,~\IEEEmembership{Member,~IEEE,} Angelos Georgios Koulouras,~\IEEEmembership{Student Member,~IEEE}
\thanks{D. Bertsimas and A. G. Koulouras are with the Operations Research Center at the Massachusetts Institute of Technology: dbertsim@mit.edu, angkoul@mit.edu. Research partially supported by the Advanced Research Projects Agency with award number DE-AR0001282.}
}

\markboth{IEEE TRANSACTIONS ON POWER SYSTEMS,~Vol.~38, No.~3, May~2023}%
{Shell \MakeLowercase{\textit{et al.}}: A Sample Article Using IEEEtran.cls for IEEE Journals}


\maketitle

\begin{abstract}
The increase of renewables in the grid and the volatility of the load create uncertainties in the day-ahead prices of electricity markets. Adaptive robust optimization (ARO) and stochastic optimization have been used to make commitment and dispatch decisions that adapt to the load and capacity uncertainty. These approaches have been successfully applied in practice but current pricing approaches used by US Independent System Operators (marginal pricing) and proposed in the literature (convex hull pricing) have two major disadvantages: a) they are deterministic in nature, that is they do not adapt to the load and capacity uncertainty, and b) require uplift payments to the generators that are typically determined by ad hoc procedures and create inefficiencies that motivate self-scheduling. In this work, we extend pay-as-bid and uniform pricing mechanisms to propose the first adaptive pricing method in electricity markets that adapts to the load and capacity uncertainty, eliminates post-market uplifts and deters self-scheduling, addressing both disadvantages. 
\end{abstract}

\begin{IEEEkeywords}
Pricing, Robust Optimization, Adaptive Optimization, Unit Commitment, Energy Markets, Uncertainty
\end{IEEEkeywords}

\section{Introduction}

The future of electricity markets is expected to prominently feature renewable energy sources, which bring volatility and unpredictability to the markets \cite{pinson2023may, morales2013integrating}. With the increase of wind and solar power comes an increase in the price volatility, among other things, which creates issues for both the market operators and the market participants. Specifically, \cite{oren2004pay} shows that as the penetration of wind power increases, the market-clearing price becomes more volatile and uncertain, making it difficult for wind power producers to forecast and bid their power accurately. In addition, changes are also expected on the consumer side. For example, the introduction of electric vehicles could change demand patterns as well as available storage options \cite{hannan2022vehicle, tan2016integration}. Also, consumers could simultaneously play the role of producers thereby becoming \say{prosumers} \cite{khorasany2020new, morstyn2018using, zhang2021strategic}.

So far, most energy markets do not directly address these challenges as they have been set up using deterministic approaches \cite{silva2022short}. However, there has been recent work in robust and stochastic optimization with promising results in addressing the unit commitment (UC) problem under uncertainty \cite{bertsimas2012adaptive, lorca2016multistage, xiong2016distributionally, zheng2014stochastic, jiang2011robust}. Adaptive robust optimization (ARO) minimizes the cost under the worst-case scenario in an uncertainty set, while stochastic optimization accounts for the probability distribution of the uncertain parameters by minimizing the expected cost. In general, ARO has been used to protect against different types of uncertainty, from contingencies and demand planning to wind energy output and has lower average dispatch and total costs, indicating better economic efficiency and significantly reduces the volatility of the total costs \cite{bertsimas2012adaptive}. In \cite{bertsimas2012adaptive}, the authors use ARO to find solutions in the UC problem that are robust to uncertain nodal injections. This work is extended in \cite{lorca2016multistage}, which deals with multistage UC and dynamic uncertainty sets related to the capacity of renewable energy sources. The authors of \cite{xiong2016distributionally} also address wind uncertainty. They propose a distributionally robust approach, where they define a family of wind power distributions and minimize the expected cost under the worst-case distribution. Our approach is based on ARO but there has also been promising work on stochastic optimization. See \cite{zheng2014stochastic, reddy2017review} for an introduction to these methods.

While the previous approaches have been successful in dealing with the volatility in energy systems and markets, the pricing methods available are still mostly deterministic \cite{dutta2017literature, mazzi2017price}. One of the more popular schemes is uniform marginal-cost pricing or \say{IP pricing}, which may result in losses for some generators \cite{o2005efficient, o2016dual}. Therefore, side-payments or uplifts are provided to these generators to make them whole, which may modify their incentives. Alternative mechanisms have been proposed, including removing the negative uplifts, such that no generator incurs a loss, and raising the commodity price above marginal cost to reduce uplifts \cite{bjorndal2008equilibrium, galiana2003reconciling}. One principled version of that is the \say{convex hull} pricing scheme, which minimizes the uplifts by langragifying the energy balance constraint and maximizing over its dual price \cite{hua2016convex, andrianesis2021computation}. Possibly only \cite{ye2016uncertainty} and \cite{fang2019introducing} discuss pricing in robust UC. However, they do not use ARO and \cite{fang2019introducing} does not offer extended theoretical results. For a detailed review of the pricing methods, see \cite{liberopoulos2016critical}.

Current pricing approaches used by Independent System Operators (ISOs) (marginal pricing) and proposed in the literature (convex hull pricing) have two major disadvantages: 
a) they are deterministic in nature, that is they do not adapt to the load and capacity uncertainty, and b) require uplift payments to the generators that are typically determined by ad hoc procedures and create inefficiencies that motivate self-scheduling. In contrast,  we extend pay-as-bid and uniform pricing mechanisms to propose the first adaptive pricing method in electricity markets that adapts to the load and capacity uncertainty, eliminates post-market uplifts and deters self-scheduling, addressing both disadvantages. 

\subsection{Contributions}

In this work, we propose the first pricing method for ARO in energy markets with load and capacity uncertainty by offering contracts contingent to the uncertainty in the data of the day-ahead problem. We summarize the main contributions:

\begin{itemize}
    \item We introduce adaptive pay-as-bid and marginal pricing contracts for UC problems featuring load and capacity uncertainty. We specify fully the day-ahead payments and provide an upper-bound on the next-day or intra-day payments based on the day-ahead commitments. The payments are functions of the load and capacity uncertainty.
    \item We show that the adaptive day-ahead pay-as-bid scheme is equivalent to the adaptive uniform pricing scheme. Also, if the worst-case uncertainty is realized in the next day, the forecasted intra-day payments are optimal and the generators are indifferent between the market schedule and their optimal schedule, eliminating self-scheduling.
    \item We display the adaptive pricing in detail on the adaptive robust version of the Scarf example, see 
    \cite{o2005efficient}, and on realistic UC problems featuring ramp constraints. We compare it to deterministic marginal and convex hull pricing and find that adaptive pricing eliminates the corrections that are necessary in deterministic day-ahead problems. 
\end{itemize}

The paper is organized as follows: in Section \ref{sec:sec_method}, we introduce the ARO formulations for UC. In Section \ref{sec:adamarkets}, we present the adaptive pricing and its theoretical properties on UC with load and capacity uncertainty. In Section \ref{sec:sec_load_example} and in Section \ref{sec:sec_load_cap_example} we provide adaptive pricing on examples with load and capacity uncertainty. In Section \ref{sec:sec_multi}, we demonstrate our method on a realistic example with ramp constraints and compare it to convex hull pricing. Finally, Section \ref{sec:ada_conclusion} summarizes our conclusions.

The notation that we use is as follows: we use bold faced characters such as $ \boldsymbol{p} $ to represent vectors and capital letters such as $\boldsymbol{V}$ to represent matrices. Also, $\boldsymbol{p}^{T}$ denotes the transpose of the column vector $\boldsymbol{p}$ and $\boldsymbol{e}_{i}$ is the $i$th unit vector. We define $ [I] = \{1, \dots, I\} $. The ${\cal L}_{1}$ norm of a vector refers to the norm $\| \boldsymbol{x} \|_{1} =  \sum_{i=1}^{I} |x_i|$, the ${\cal L}_{2}$ norm refers to $\| \boldsymbol{x} \|_{2} = \sqrt{\sum_{i=1}^{I} x_{i}^{2}}$ and ${\cal L}_{\infty}$ refers to $\| \boldsymbol{x} \|_{\infty} = \max_{i \in [I]} |x_i|$.

\section{Adaptive Robust UC}
\label{sec:sec_method}

In this section, we introduce the ARO formulation for an energy market robust to load and capacity uncertainty. The formulation is based on the popular Scarf example, see \cite{o2005efficient}, but considers uncertainty in the load and capacity parameters. Reserves, transmission and ramp constraints can be added using linear constraints with small changes.

Consider the following example which tries to minimize the total cost of meeting a fixed level of demand.
$$  
    {\everymath{\displaystyle}
    \scalebox{0.99}{
    $
    \begin{array}{rlr}
        \min_{x, p} & \sum_{i=1}^{I} F_{i} x_{i} + C_{i} p_{i} \\
        \text{s.t.} & \sum_{i=1}^{I} p_{i} = \sum_{j=1}^{J} \bar{q}_{j}, \\
        & p_{i} \leq \bar{p}^{\max}_{i} x_{i}, \quad \forall i \in [I], \\
        & p_{i} \geq 0, \quad \forall i \in [I], \\
        & x_{i} \in \{0, 1\}, \quad \forall i \in [I]. \\
    \end{array}
    $
    }
    }
$$
We have $I$ generators, each with a turn-on cost of $F_{i}$ and a unit production cost of $C_{i}$. The expected maximum production levels of generator $i$ are $\bar{p}_{i}^{\max}$. We also have $J$ demand nodes, with expected demand or load $\bar{q}_{j}$ at each node $j$. The binary variable $x_{i}=1$, if generator $i$ is turned on, otherwise $x_{i}=0$. Note, the variable $p_{i}$ represents the dispatch of generator $i$.

We can expand on this model by considering uncertainty in the load $\boldsymbol{q}$ and in the capacity $\boldsymbol{p}^{\max}$. In ARO, we describe the uncertainty in the parameters with an uncertainty set that contains all scenarios against which we want to be robust. Essentially, the constraints in our problem should be satisfied for all possible realizations of $\boldsymbol{q}$ and $\boldsymbol{p}^{\max}$ in their uncertainty sets \cite{ben2009robust, bertsimas2011theory, bertsimas2022robust}. We consider uncertainty sets where the norm of the residuals from the expected load $\boldsymbol{d} = \boldsymbol{q} - \bar{\boldsymbol{q}}$ is at most $\Gamma_{q}$ and where the norm of residuals $\boldsymbol{r} = \boldsymbol{p}^{\max}-\bar{\boldsymbol{p}}^{\max}$ from the expected capacity is at most $\Delta_{p}$. 
$$
{\everymath{\displaystyle}
    \mathcal{D} = \{\|\boldsymbol{d}\|_{\ell} \leq \Gamma_{q}\}, \; \; \;
        \mathcal{U} = \{\|\boldsymbol{r}\|_{\ell} \leq \Delta_p\}.
}
$$
with $\ell \in \{1, 2, \infty \}$ corresponding to the budget, ellipsoidal and box uncertainty sets. The values of $\Gamma_{q}$ and $\Delta_{p}$ control the conservativeness of our formulation \cite{guan2013uncertainty}.

In ARO, the second-stage decisions $\boldsymbol{p}$ are functions of both uncertain parameters $\boldsymbol{d}, \boldsymbol{r}$ and the problem that minimizes the worst-case commitment and dispatch cost is
\begin{equation*}
    {\everymath{\displaystyle}
    \scalebox{0.99}{
    $
        \begin{array}{rlr}
            \min_{\boldsymbol{x}, \boldsymbol{p}} \max_{\boldsymbol{d} \in \mathcal{D}, \boldsymbol{r} \in \mathcal{U}}  & \sum_{i=1}^{I} F_{i} x_{i} + C_{i} p_{i}(\boldsymbol{d}, \boldsymbol{r}) \\
            \text{s.t.} & \sum_{i=1}^{I} p_{i}(\boldsymbol{d}, \boldsymbol{r}) \geq \sum_{j=1}^{J} (d_{j} + \bar{q}_{j}), \quad \forall\boldsymbol{d} \in \mathcal{D},\\
            & p_{i}(\boldsymbol{d}, \boldsymbol{r}) \leq (r_{i} + p^{\max}_{i}) x_{i},~ \forall i \in [I], ~\boldsymbol{r} \in \mathcal{U}, \\
            & p_{i}(\boldsymbol{d}, \boldsymbol{r}) \geq 0, \quad \forall i \in [I], \\
            & x_{i} \in \{0, 1\}, \quad \forall i \in [I]. \\
        \end{array}
    $
    }
    }
\end{equation*}

The optimization variable $p_{i} = p_{i}(\boldsymbol{d}, \boldsymbol{r})$, called a decision rule, is in fact a vector function \cite{bertsimas2022robust}. In this paper, to make the problem tractable, we restrict $p_{i}(\boldsymbol{d}, \boldsymbol{r})$ to linear functions or linear decision rules (LDR). Such a decision rule may not be optimal, because of the restriction to a certain class, but LDR have shown very good performance in practice and are optimal in many settings \cite{dehghan2017adaptive, bertsimas2022robust, jabr2017linear}. Alternatively, we can use decomposition schemes to learn the decision rule implicitly \cite{bertsimas2012adaptive, zeng2013solving}.

Using LDR, the second-stage decisions $\boldsymbol{p}$ are linear functions of the uncertain parameters $\boldsymbol{d}, \boldsymbol{r}$. So, we have an $I$-dimensional vector $\boldsymbol{u}$, an $I \times J$ matrix $\boldsymbol{V}$, an $I \times I$ matrix $\boldsymbol{Z}$ and $\boldsymbol{p}(\boldsymbol{d}, \boldsymbol{r}) = \boldsymbol{u} + \boldsymbol{V} \boldsymbol{d} + \boldsymbol{Z} \boldsymbol{r}$ or, for each $i$, $p_{i}(\boldsymbol{d}, \boldsymbol{r}) = u_{i} + \sum_{j=1}^{J} V_{ij} {d}_{j} + \sum_{k=1}^{I} Z_{ik} r_{k}$. The dispatch $\boldsymbol{p}$ contains a non-adaptive part $\boldsymbol{u}$ and an adaptive part that depends on $\boldsymbol{d}$ and $\boldsymbol{r}$. We will use these terms for the rest of the paper. If we set $\boldsymbol{V}=\boldsymbol{0}$ and $\boldsymbol{Z}=\boldsymbol{0}$, the problem becomes an RO problem. Using LDR, the previous formulation is equivalent to
\begin{equation}
    \label{centralized1_capacity}
    {\everymath{\displaystyle}
    \scalebox{0.85}{
    $
        \begin{array}{rlr}
           \min_{\boldsymbol{x}, \boldsymbol{u}, \boldsymbol{V}, \boldsymbol{Z}} \max_{\boldsymbol{d} \in \mathcal{D}, \boldsymbol{r} \in \mathcal{U}} & \sum_{i=1}^{I} F_{i} x_{i} + C_{i} (u_{i} + \sum_{j=1}^{J} V_{ij} {d}_{j} + \sum_{k=1}^{I} Z_{ik} r_{k}) \\
            \text{s.t.} & \sum_{i=1}^{I} (u_{i} + \sum_{j=1}^{J} V_{ij} {d}_{j} + \sum_{k=1}^{I} Z_{ik} r_{k}) \geq \sum_{j=1}^{J} (d_{j} + \bar{q}_{j}), \\
            & \sum_{i=1}^{I} u_{i} = \sum_{j=1}^{J} \bar{q}_{j}, \\
            & u_{i} + \sum_{j=1}^{J} V_{ij} {d}_{j} + \sum_{k=1}^{I} Z_{ik} r_{k} \leq (r_{i} + \bar{p}^{\max}_{i}) x_{i}, \quad \forall i, \\
            & u_{i} + \sum_{j=1}^{J} V_{ij} {d}_{j} + \sum_{k=1}^{I} Z_{ik} r_{k} \geq 0, \quad \forall i, \\
            & x_{i} \in \{0, 1\}, \quad \forall i, \\
        \end{array}
    $
    }
    }
\end{equation}
where constraints involving $\boldsymbol{d},  \boldsymbol{r} $ hold for all 
$\boldsymbol{d} \in \mathcal{D}$ and $\boldsymbol{r} \in \mathcal{U}$. 

Let the objective function of Problem \eqref{centralized1_capacity} be  $\xi^*$. 
The solution $\boldsymbol{x}^{*}$ and $\boldsymbol{p}^{*}(\boldsymbol{d}, \boldsymbol{r}) = \boldsymbol{u}^{*} + \boldsymbol{V}^{*} \boldsymbol{d} + \boldsymbol{Z}^{*} \boldsymbol{r}$ satisfies the constraints for any realization of $\boldsymbol{d} \in \mathcal{D}$ and of $\boldsymbol{r} \in \mathcal{U}$. For example, the second constraint ensures that the total production will be more than the total demand for all scenarios in the uncertainty sets. In addition, to make pricing more intuitive, we have included the third constraint, which ensures that the total non-adaptive dispatch $\sum_{i=1}^{I} u_{i}$ is equal to the expected load. The previous formulation is equivalent to:
\begin{equation*}
    {\everymath{\displaystyle}
    \scalebox{0.85}{
    $
        \begin{array}{rlr}
            \min_{\boldsymbol{x}, \boldsymbol{u}, \boldsymbol{V}, \boldsymbol{Z}, \eta} & \sum_{i=1}^{I} F_{i} x_{i} + \eta \\
            \text{s.t.} & \sum_{i=1}^{I} C_{i}u_{i} + \max_{\boldsymbol{d} \in \mathcal{D}}  \sum_{i=1}^{I} C_{i} \sum_{j=1}^{J} V_{ij} d_{j} + \max_{\boldsymbol{r} \in \mathcal{U}}  \sum_{i=1}^{I} C_{i} \sum_{k=1}^{I} Z_{ik} r_{k} \leq \eta, \\
            & \max_{\boldsymbol{d} \in \mathcal{D}} \sum_{j=1}^{J} (1 - \sum_{i=1}^{I} V_{ij}){d}_{j} + \max_{\boldsymbol{r} \in \mathcal{U}} \sum_{i=1}^{I} \sum_{k=1}^{I} - Z_{ik} r_{k} \leq 0, \\
            & \sum_{i=1}^{I} u_{i} = \sum_{j=1}^{J} \bar{q_{j}}, \\
            & u_{i} + \max_{\boldsymbol{d} \in \mathcal{D}} \sum_{j=1}^{J} V_{ij} {d}_{j} + \max_{\boldsymbol{r} \in \mathcal{U}} \{ - x_{i} r_{i} + \sum_{k=1}^{I} Z_{ik} r_{k} \} \leq \bar{p}^{\max}_{i} x_{i},  \forall i, \\
            & - u_{i} + \max_{\boldsymbol{d} \in \mathcal{D}} \sum_{j=1}^{J} -V_{ij} d + \max_{\boldsymbol{r} \in \mathcal{U}} \sum_{k=1}^{I} - Z_{ik} r_{k} \leq 0, \quad \forall i, \\
            & x_{i} \in \{0, 1\}, \quad \forall i, \\
        \end{array}
    $
    }
    }
\end{equation*}
or, by using the robust counterpart with $\mathcal{D} = \{\|\boldsymbol{d}\|_{\ell} \leq \Gamma_{q}\}$ and $\mathcal{U} = \{\|\boldsymbol{r}\|_{\ell} \leq \Delta_{p}\}$, where $\ell^{*}$ is the dual norm of $\ell$, and $\boldsymbol{V}_{i}$ is the $i$-th row of $\boldsymbol{V}$ and $\boldsymbol{Z}_{i}$ is the $i$-th row of $\boldsymbol{Z}$,
\begin{equation*}
    {\everymath{\displaystyle}
    \scalebox{0.85}{
    $
          \begin{array}{rlr}
             \min_{\boldsymbol{x}, \boldsymbol{u}, \boldsymbol{V}, \boldsymbol{Z}, \eta} & \sum_{i=1}^{I} F_{i} x_{i} + \eta \\
            \text{s.t.} & \sum_{i=1}^{I} C_{i}u_{i} + \Gamma_{q} \|\sum_{i=1}^{I} C_{i} \boldsymbol{V}_{i}\|_{\ell^{*}} + \Delta_{p} \|\sum_{i=1}^{I} C_{i} \boldsymbol{Z}_{i}\|_{\ell^{*}} \leq \eta, \\
             & \Gamma_{q} \|\boldsymbol{1} - \sum_{i=1}^{I} \boldsymbol{V}_{i} \|_{\ell^{*}} + \Delta_{p} \| - \sum_{i=1}^{I} \boldsymbol{Z}_{i} \|_{\ell^{*}} \leq 0, \\
             & \sum_{i=1}^{I} u_{i} = \sum_{j=1}^{J} \bar{q_{j}}, \\
             & u_{i} + \Gamma_{q} \| \boldsymbol{V}_{i} \|_{\ell^{*}} + \Delta_{p} \| x_{i} \boldsymbol{e}_{i} - \boldsymbol{Z}_{i} \|_{\ell^{*}} \leq p^{\max}_{i} x_{i}, \quad \forall i, \\
             & - u_{i} + \Gamma_{q} \| \boldsymbol{V}_{i} \|_{\ell^{*}} + \Delta_{p} \|\boldsymbol{Z}_{i} \|_{\ell^{*}} \leq 0, \forall i, \\
             & x_{i} \in \{0, 1\}, \quad \forall i. \\
        \end{array}
       $ 
        }    
        }
\end{equation*}

We are going to work with the following version of the previous formulation, because we want to use the dual problem as well. Following the example of \cite{o2005efficient}, we also set the binary variables to their optimal values $\boldsymbol{x}^{*}$.
\begin{equation}
\label{centralized2_capacity}
    {\everymath{\displaystyle}
    \scalebox{0.85}{
    $
       \begin{array}{rlrr}
            \min_{\boldsymbol{x}, \boldsymbol{u}, \boldsymbol{V}, \boldsymbol{Z}, \eta} & \sum_{i=1}^{I} F_{i} x_{i} + \eta \\
             \text{s.t.} & \sum_{i=1}^{I} C_{i}u_{i} + \Gamma_{q} \| \boldsymbol{\omega} \|_{\ell^{*}} + \Delta_{p} \| \bar{\boldsymbol{\omega}} \|_{\ell^{*}} \leq \eta, & \nu \\
             & \boldsymbol{\omega} = \sum_{i=1}^{I} C_{i} \boldsymbol{V}_{i}, \; \; \bar{\boldsymbol{\omega}} = \sum_{i=1}^{I} C_{i} \boldsymbol{Z}_{i}, & \boldsymbol{\alpha}, \bar{\boldsymbol{\alpha}} \\
             & \Gamma_{q} \|\boldsymbol{\tau}\|_{\ell^{*}} + \Delta_{p} \|\bar{\boldsymbol{\tau}}\|_{\ell^{*}} \leq 0, & \lambda \\
            & \boldsymbol{\tau} = \boldsymbol{1} - \sum_{i=1}^{I} \boldsymbol{V}_{i}, \; \; \bar{\boldsymbol{\tau}} = - \sum_{i=1}^{I} \boldsymbol{Z}_{i}, & \boldsymbol{\theta}, \bar{\boldsymbol{\theta}}\\
             & \sum_{i=1}^{I} u_{i} = \sum_{j=1}^{J} \bar{q_{j}}, & \mu \\
             & u_{i} + \Gamma_{q} \| \boldsymbol{\psi}_{i} \|_{\ell^{*}} + \Delta_{p} \| \bar{\boldsymbol{\psi}}_{i} \|_{\ell^{*}} \leq p^{\max}_{i} x_{i}, \quad \forall i, & \sigma_{i} \\
             & \boldsymbol{\psi}_{i} = \boldsymbol{V}_{i}, \; \; \bar{\boldsymbol{\psi}}_{i} = x_{i} \boldsymbol{e}_{i} - \boldsymbol{Z}_{i}, & \boldsymbol{\beta}_{i}, \bar{\boldsymbol{\beta}}_{i} \\
             & - u_{i} + \Gamma_{q} \| \boldsymbol{\phi}_{i}  \|_{\ell^{*}} + \Delta_{p} \| \bar{\boldsymbol{\phi}}_{i}  \|_{\ell^{*}} \leq 0, \quad \forall i, & \zeta_{i} \\
             & \boldsymbol{\phi}_{i} = \boldsymbol{V}_{i}, \; \; \bar{\boldsymbol{\phi}}_{i} = \boldsymbol{Z}_{i}, & \boldsymbol{\gamma}_{i}, \bar{\boldsymbol{\gamma}}_{i} \\
             & x_{i} = x_{i}^{*}, \quad \forall i, & \rho_{i} \\
             & x_{i} \geq 0, \quad \forall i, \\
        \end{array}
    $
    }
    }
\end{equation}
where the right column contains the dual variables for the corresponding constraints. The Lagrangian of the previous problem is
\begin{equation*}
    {\everymath{\displaystyle}
    \scalebox{0.90}{
    $
        \begin{array}{llr}
            & \mathcal{L} = \mu \sum_{j=1}^{J} \bar{q_{j}} + \sum_{j=1}^{J} \theta_{j} + \sum_{i=1}^{I} x_{i}^{*} \rho_{i} \\
            & + (1-\nu) \eta \\
            & + \sum_{i=1}^{I} (F_{i} - \sigma_{i} p^{\max}_{i} - \rho_{i} + \bar{\beta}_{ii}) x_{i} \\
            & + \sum_{i=1}^{I} (\nu C_{i} - \mu + \sigma_{i} - \zeta_{i}) u_{i} \\
            & + \sum_{i=1}^{I}(C_{i} \boldsymbol{\alpha} - \boldsymbol{\theta} + \boldsymbol{\beta}_{i} + \boldsymbol{\gamma}_{i})^{T} \boldsymbol{V}_{i} \\
            & + \sum_{i=1}^{I}(C_{i} \bar{\boldsymbol{\alpha}} -  \bar{\boldsymbol{\theta}} - \bar{\boldsymbol{\beta}}_{i} + \bar{\boldsymbol{\gamma}}_{i})^{T} \boldsymbol{Z}_{i} \\
            & - (\boldsymbol{\alpha}^{T} \boldsymbol{\omega} - \nu \Gamma_{q} \| \boldsymbol{\omega} \|_{\ell^{*}}) 
            - (\bar{\boldsymbol{\alpha}}^{T} \bar{\boldsymbol{\omega}} - \nu \Delta_{p} \| \bar{\boldsymbol{\omega}} \|_{\ell^{*}}) \\
            & - (\boldsymbol{\theta}^{T} \boldsymbol{\tau} - \lambda \Gamma_{q} \| \boldsymbol{\tau} \|_{\ell^{*}})
            - (\bar{\boldsymbol{\theta}}^{T} \bar{\boldsymbol{\tau}} - \lambda \Delta_{p} \| \bar{\boldsymbol{\tau}} \|_{\ell^{*}})\\
            & \sum_{i=1}^{I} - (\boldsymbol{\beta}_{i}^{T} \boldsymbol{\psi}_{i} - \sigma_{i} \Gamma_{q} \| \boldsymbol{\psi}_{i} \|_{\ell^{*}})
            +  \sum_{i=1}^{I} - (\bar{\boldsymbol{\beta}}_{i}^{T} \bar{\boldsymbol{\psi}}_{i} - \sigma_{i} \Delta_{p} \| \bar{\boldsymbol{\psi}}_{i} \|_{\ell^{*}}) \\
            & \sum_{i=1}^{I} - (\boldsymbol{\gamma}_{i}^{T} \boldsymbol{\phi}_{i} - \zeta_{i} \Gamma_{q} \| \boldsymbol{\phi}_{i} \|_{\ell^{*}})
             + \sum_{i=1}^{I} - (\bar{\boldsymbol{\gamma}}_{i}^{T} \bar{\boldsymbol{\phi}_{i}} - \zeta_{i} \Delta_{p} \| \bar{\boldsymbol{\phi}}_{i} \|_{\ell^{*}}). \\
        \end{array}
    $
    }
    }
\end{equation*}

By taking the minimum of the Lagrangian $\mathcal{L}$ over the primal variables and then maximizing over the dual variables, we obtain the dual problem \cite{bertsekas1997nonlinear}. Note that $\min_{\boldsymbol{\omega}} -(\boldsymbol{\alpha}^{T} \boldsymbol{\omega} - \nu \Gamma_{q} \| \boldsymbol{\omega} \|_{\ell^{*}}) = 0$ if $\| \boldsymbol{\alpha} \|_{\ell} \leq \Gamma_{q} \nu$, otherwise it is $-\infty$. The same holds for the other similar terms in the Lagrangian. So, at optimality, for all $i$,
\begin{equation}
    \label{dual_norm_alpha}
    {\everymath{\displaystyle}
    \scalebox{0.85}{
    $
    (\boldsymbol{\alpha}^{*})^{T} \boldsymbol{\omega}^{*} = \nu^{*} \Gamma_{q} \| \boldsymbol{\omega}^{*} \|_{\ell^{*}}, \; \;
    (\bar{\boldsymbol{\alpha}}^{*})^{T} \bar{\boldsymbol{\omega}}^{*} = \nu^{*} \Delta_{p} \| \bar{\boldsymbol{\omega}}^{*} \|_{\ell^{*}},
    $
    }
    }
\end{equation}
\begin{equation}
    \label{dual_norm_theta}
    {\everymath{\displaystyle}
    \scalebox{0.90}{
    $
    (\boldsymbol{\theta}^{*})^{T} \boldsymbol{\tau}^{*} = \lambda^{*} \Gamma_{q} \| \boldsymbol{\tau}^{*} \|_{\ell^{*}}, \; \;
    (\bar{\boldsymbol{\theta}}^{*})^{T} \bar{\boldsymbol{\tau}}^{*} = \lambda^{*} \Delta_{p} \| \bar{\boldsymbol{\tau}}^{*} \|_{\ell^{*}},
    $
    }
    }
\end{equation}
\begin{equation}
    \label{dual_norm_beta}
    {\everymath{\displaystyle}
    \scalebox{0.85}{
    $
    (\boldsymbol{\beta}_{i}^{*})^{T} \boldsymbol{\psi}_{i}^{*} = \sigma_{i}^{*} \Gamma_{q} \| \boldsymbol{\psi}_{i}^{*} \|_{\ell^{*}}, \; \;
    (\bar{\boldsymbol{\beta}}_{i}^{*})^{T} \bar{\boldsymbol{\psi}}_{i}^{*} = \sigma_{i}^{*} \Delta_{p} \| \bar{\boldsymbol{\psi}}_{i}^{*} \|_{\ell^{*}},
    $
    }
    }
\end{equation}
\begin{equation}
    \label{dual_norm_gamma}
    {\everymath{\displaystyle}
    \scalebox{0.85}{
    $
    (\boldsymbol{\gamma}_{i}^{*})^{T} \boldsymbol{\phi}_{i}^{*} = \zeta_{i}^{*} \Gamma_{q} \| \boldsymbol{\phi}_{i}^{*} \|_{\ell^{*}}, \; \;
    (\bar{\boldsymbol{\gamma}}_{i}^{*})^{T} \bar{\boldsymbol{\phi}}_{i}^{*} = \zeta_{i}^{*} \Delta_{p} \| \bar{\boldsymbol{\phi}}_{i}^{*} \|_{\ell^{*}}.
    $
    }
    }
\end{equation}

The corresponding dual problem is
\begin{equation}
\label{centralizeddual_capacity}
    {\everymath{\displaystyle}
        \scalebox{0.95}{
        $
         \begin{array}{rlr}
             \max & \mu \sum_{j=1}^{J} \bar{q_{j}} + \sum_{j=1}^{J} \theta_{j} + \sum_{i=1}^{I} x_{i}^{*} \rho_{i}, \\
             \text{s.t.} & F_{i} \geq \rho_{i} + \sigma_{i} p^{\max}_{i} - \bar{\beta}_{ii}, \quad \forall i, & x_{i} \\
             & \nu = 1, & \eta \\
             & \nu C_{i} = \mu - \sigma_{i} + \zeta_{i}, \quad \forall i, & u_{i} \\
             & C_{i} \alpha_{j} = \theta_{j} - \beta_{ij} - \gamma_{ij}, \quad \forall i,j, & V_{ij} \\
            & C_{i} \bar{\alpha}_{k} = \bar{\theta}_{k} + \bar{\beta}_{ik} - \bar{\gamma}_{ik}, \quad \forall i,k, & Z_{ik} \\
             & \| \boldsymbol{\alpha} \|_{\ell} \leq \Gamma_{q} \nu, \quad \| \bar{\boldsymbol{\alpha}} \|_{\ell} \leq \Delta_{p} \nu \\ 
             & \| \boldsymbol{\theta}\|_{\ell} \leq \Gamma_{q} \lambda, \quad \| \bar{\boldsymbol{\theta}} \|_{\ell} \leq \Delta_{p} \lambda \\ 
             & \| \boldsymbol{\beta}_{i} \|_{\ell} \leq \Gamma_{q} \sigma_{i}, \quad \| \bar{\boldsymbol{\beta}}_{i} \|_{\ell} \leq \Delta_{p} \sigma_{i}, \quad \forall i, \\ 
             & \| \boldsymbol{\gamma}_{i} \|_{\ell} \leq \Gamma_{q} \zeta_{i}, \quad \| \bar{\boldsymbol{\gamma}}_{i} \|_{\ell} \leq \Delta_{p} \zeta_{i} \quad \forall i, \\
             & \lambda, \sigma_{i}, \zeta_{i} \geq 0, \quad \forall i.
        \end{array}
        $
        }
    }
\end{equation}

We use $(\cdot)^{*}$ to denote the optimal solutions to Problem \eqref{centralized2_capacity} and Problem \eqref{centralizeddual_capacity}. We will use this notation for the rest of this work.


\section{Adaptive Pricing}
\label{sec:adamarkets}

In this section, we introduce adaptive pricing and provide its theoretical properties.

We suggest that the day-ahead payments include only the commitment and non-adaptive dispatch costs. So, the day-ahead payments reflect the cost of meeting the expected load while planning for the worst-case scenario. The adaptive part of the dispatch serves as an upper bound on the intra-day payments, which take place the following day based on the economic dispatch problems. This upper bound is also equal to the optimal intra-day payments, when the worst-case uncertainty is realized.

\begin{table*}[!htb]
\caption{Pay-as-bid and marginal pricing with uncertain load for the budget, ellipsoidal and box uncertainty sets. The uplifts in marginal pricing depend on the size $\Gamma_{q}$ of the uncertainty set.}
    {\everymath{\displaystyle}
\begin{tabular*}{\textwidth}{@{\extracolsep{\fill}}ccccc}
    \vspace{5pt}
  \textbf{Uncertainty Set} & $\mathcal{D}$ & $\max_{d \in \mathcal{D}} \sum_{j=1}^{J} V_{ij}^{*} d_{j}$ & \textbf{Pay-as-bid} & \textbf{Uniform} \\
  \hline
  \\
  \vspace{5pt}
  Budget & $\{\|\boldsymbol{d}\|_{1} \leq \Gamma_{q}\}$ & $\Gamma_{q} \| V_{i}^{*}\|_{\infty}$ & $F_{i} x_{i}^{*} + C_{i} u_{i}^{*}$ & $\mu^{*} u_{i}^{*} + \rho_{i}^{*} x_{i}^{*} + \sigma_{i}^{*} \Gamma_{q} \| V_{i}^{*}\|_{\infty} + \zeta_{i}^{*} \Gamma_{q} \| V_{i}^{*}\|_{\infty}$ \\
  \vspace{5pt}
  Ellipsoidal & $\{\|\boldsymbol{d}\|_{2} \leq \Gamma_{q}\}$ & $\Gamma_{q} \| V_{i}^{*}\|_{2}$ & $F_{i} x_{i}^{*} + C_{i} u_{i}^{*}$ & $\mu^{*} u_{i}^{*} + \rho_{i}^{*} x_{i}^{*} + \sigma_{i}^{*} \Gamma_{q} \| V_{i}^{*} \|_{2} + \zeta_{i}^{*} \Gamma_{q} \| V_{i}^{*} \|_{2}$ \\
  \vspace{5pt}
  Box & $\{ \|\boldsymbol{d}\|_{\infty} \leq \Gamma_{q}\}$ & $\Gamma_{q} \| V_{i}^{*}\|_{1}$ & $F_{i} x_{i}^{*} + C_{i} u_{i}^{*}$ & $\mu^{*} u_{i}^{*} + \rho_{i}^{*} x_{i}^{*} + \sigma_{i}^{*} \Gamma_{q} \| V_{i}^{*} \|_{1} + \zeta_{i}^{*} \Gamma_{q} \| V_{i}^{*} \|_{1}$ \\
  \\
\end{tabular*}
\label{tab:tab_unc}
}
\end{table*}

\paragraph{Pay-as-bid pricing} The day-ahead payments to generator $i$ are based on their bids $F_{i}, C_{i}$ and the non-adaptive part of the dispatch. Specifically, each generator is paid
$$
    {\everymath{\displaystyle}
    \scalebox{0.9}{
    $
        F_{i} x_{i}^{*} + C_{i} u_{i}^{*}.
    $
    }
    }
$$

\paragraph{Marginal pricing} The price for the non-adaptive dispatch is $\mu$ and each generator is paid some uplift in a discriminatory way. So, each generator $i$ is paid
$$
    {\everymath{\displaystyle}
    \scalebox{0.9}{
    $
    \begin{array}{cc}
         \vspace{5pt}
         & \mu^{*} u_{i}^{*} + (\rho_{i}^{*} - \bar{\beta}_{ii}^{*}) x_{i}^{*} \\
         \vspace{5pt}
         & + \sigma_{i}^{*} \; (\Gamma_{q} \| \boldsymbol{V}_{i}^{*}\|_{\ell^{*}} + \Delta_p \|x_{i}^{*} \boldsymbol{e}_{i} - \boldsymbol{Z}_{i}^{*}\|_{\ell^{*}}) \\
         & + \zeta_{i}^{*} \; (\Gamma_{q} \|  \boldsymbol{V}_{i}^{*} \|_{\ell^{*}} + \Delta_p \|  \boldsymbol{Z}_{i}^{*} \|_{\ell^{*}}).
    \end{array}
    $
    }
    }
$$

In contrast to deterministic pricing, we also price the uncertainty by introducing payments based on sizes $\Gamma_{q}$ and $\Delta_{p}$ of the uncertainty sets. Note that if we set $\Delta_{p} = 0$ in the uncertainty set $\mathcal{U}$, we do not consider the capacity uncertainty. Table \ref{tab:tab_unc} summarizes the payments when there is only load uncertainty or $\Delta_{p}=0$. Similarly, if we set $\Gamma_{q}=0$, we do not consider uncertainty in the load.

\subsection{Pay-As-Bid and Uniform Pricing Equivalence}

One important theoretical property of our approach is that the day-ahead pay-as-bid and marginal pricing payments are the same. This result is presented more formally in the following theorem.

\begin{thm}
The pay-as-bid payment   $F_{i} x_{i}^{*} + C_{i} u_{i}^{*}$ and  the uniform price payment 
to each generator $i$ 
$$
    {\everymath{\displaystyle}
    \scalebox{0.9}{
    $
    \begin{array}{cc}
         \vspace{5pt}
         & \mu^{*} u_{i}^{*} + (\rho_{i}^{*} - \bar{\beta}_{ii}^{*}) x_{i}^{*} \\
         \vspace{5pt}
         & + \sigma_{i}^{*} \; (\Gamma_{q} \| \boldsymbol{V}_{i}^{*}\|_{\ell^{*}} + \Delta_p \|x_{i}^{*} \boldsymbol{e}_{i} - \boldsymbol{Z}_{i}^{*}\|_{\ell^{*}}) \\
         & + \zeta_{i}^{*} \; (\Gamma_{q} \|  \boldsymbol{V}_{i}^{*} \|_{\ell^{*}} + \Delta_p \|  \boldsymbol{Z}_{i}^{*} \|_{\ell^{*}})
    \end{array}
    $
    }
    }
$$
are equal.
\end{thm}

\begin{proof}
    Using complementary slackness between Problems \eqref{centralized2_capacity} and \eqref{centralizeddual_capacity}, for all $i \in [I]$,
$$
        \scalebox{0.9}{
        $
        F_{i} x_{i}^{*} = \rho_{i}^{*} x_{i}^{*} + \sigma_{i}^{*} p^{\max}_{i} x_{i}^{*} - \bar{\beta}_{ii}^{*} x_{i}^{*},
        $
        }
$$
$$
        \scalebox{0.9}{
        $
         C_{i} u_{i}^{*} = \mu^{*} u_{i}^{*} - \sigma_{i}^{*} u_{i}^{*} + \zeta_{i}^{*} u_{i}^{*},
        $
        }
$$
$$      
        \scalebox{0.9}{
        $
        \zeta_{i}^{*} u_{i}^{*} - \zeta_{i}^{*} \Gamma_{q} \| \boldsymbol{\phi}_{i}^{*}  \|_{\ell^{*}} - \zeta_{i}^{*} \Delta_{p} \| \bar{\boldsymbol{\phi}}_{i}^{*}  \|_{\ell^{*}} = 0,
        $
        }
$$
$$
        \scalebox{0.9}{
        $
        \sigma_{i}^{*} p^{\max}_{i} x_{i}^{*} - \sigma_{i}^{*} u_{i}^{*} - \sigma_{i}^{*} \Gamma_{q} \| \boldsymbol{\psi}_{i}^{*} \|_{\ell^{*}} - \sigma_{i}^{*} \Delta_{p} \| \bar{\boldsymbol{\psi}}_{i}^{*}  \|_{\ell^{*}}  = 0.
        $
        }
$$

    So, for all $i \in [I]$,

$$
{\everymath{\displaystyle}
        \scalebox{0.9}{
        $
        \begin{array}{l}
            \vspace{10pt}
            F_{i}x_{i}^{*} + C_{i} u_{i}^{*} \\
            \vspace{10pt}
             =  \rho_{i}^{*} x_{i}^{*} + \sigma_{i}^{*} p^{\max}_{i} x_{i}^{*} - \bar{\beta}_{ii}^{*} x_{i}^{*}
             + \mu^{*} u_{i}^{*} - \sigma_{i}^{*} u_{i}^{*} + \zeta_{i}^{*} u_{i}^{*} \\
             \vspace{10pt}
             = (\rho_{i}^{*} - \bar{\beta}_{ii}^{*}) x_{i}^{*} + \mu u_{i}^{*}
             + \sigma_{i}^{*} (p^{\max}_{i}x_{i}^{*} - u_{i}^{*})  + \zeta_{i}^{*} u_{i}^{*} \\
             \vspace{10pt}
             = (\rho_{i}^{*} - \bar{\beta}_{ii}^{*}) x_{i}^{*} + \mu u_{i}^{*} \\
             \vspace{10pt}
             + \sigma_{i}^{*} \Gamma_{q} \| \boldsymbol{\psi}_{i}^{*}  \|_{\ell^{*}} + \sigma_{i}^{*} \Delta_{p} \| \bar{\boldsymbol{\psi}}_{i}^{*}  \|_{\ell^{*}} 
             + \zeta_{i}^{*} \Gamma_{q} \| \boldsymbol{\phi}_{i}^{*}  \|_{\ell^{*}} + \zeta_{i}^{*} \Delta_{p} \| \bar{\boldsymbol{\phi}}_{i}^{*}  \|_{\ell^{*}}  \\
             \vspace{10pt}
             = \mu^{*} u_{i}^{*} + (\rho_{i}^{*} - \bar{\beta}_{ii}^{*}) x_{i}^{*} \\
             + \sigma_{i}^{*} \Gamma_{q} \| \boldsymbol{V}_{i}^{*}\|_{\ell^{*}} + \sigma_{i}^{*} \Delta_{p} \| x_{i}^{*} \boldsymbol{e}_{i} - \boldsymbol{Z}_{i}^{*}\|_{\ell^{*}}
             + \zeta_{i}^{*} \Gamma_{q} \| \boldsymbol{V}_{i}^{*}\|_{\ell^{*}} + \zeta_{i}^{*} \Delta_{p} \| \boldsymbol{Z}_{i}^{*}\|_{\ell^{*}}.
        \end{array}
        $
    }
}
$$

The last equality follows from the constraints $\boldsymbol{\psi}_{i} = \boldsymbol{V}_{i}$, $\boldsymbol{\phi}_{i} = \boldsymbol{V}_{i}$, $\bar{\boldsymbol{\psi}}_{i} = x_{i}\boldsymbol{e}_{i} - \boldsymbol{Z}_{i}$ and $\bar{\boldsymbol{\phi}}_{i} = \boldsymbol{Z}_{i}$ of Problem \eqref{centralized2_capacity}. \end{proof}

\subsection{Total Worst-Case Cost Equivalence}

In the previous section, we considered only the day-ahead payments. In this section, we show that when the worst-case uncertainty is realized, the pay-as-bid and marginal pricing payments are the same.
Let $\boldsymbol{d}^{*}= \arg \max_{\boldsymbol{d} \in \mathcal{D}}  \sum_{i=1}^{I} C_{i} \sum_{j=1}^{J} V_{ij}^* d_{j}$ and $\boldsymbol{r}^{*}= \arg \max_{\boldsymbol{r} \in \mathcal{U}}  \sum_{i=1}^{I} C_{i} \sum_{k=1}^{I} Z_{ik}^{*} r_{k}$
be the worst case realization of the residual load $\boldsymbol{d}$ and the residual capacity $\boldsymbol{r}$ respectively for the objective function of Problem \eqref{centralized1_capacity}.

\begin{thm}
\label{thm:thm2_capacity}
If the worst-case uncertainty $\boldsymbol{d}^{*}, \boldsymbol{r}^{*}$ is realized, then the total pay-as-bid payment to generator $i$
$$
       f_i:=F_{i} x_{i}^{*} + C_{i} u_{i}^{*} + C_{i} \sum_{j=1}^{J} V_{ij}^{*} d_{j}^{*} + C_{i} \sum_{k=1}^{I} Z_{ik}^{*} r_{k}^{*},
$$
and the uniform price payment to generator $i$ 
$$
       g_i:= \rho_{i}^{*} x_{i}^{*} + \mu^{*} u_{i}^{*} + \sum_{j=1}^{J} \theta_{j}^{*} V_{ij}^{*} + \sum_{k=1}^{I} \bar{\theta}_{k}^{*} Z_{ik}^{*}.
$$
are equal and satisfy $\sum_{i=1}^{I} f_i=\sum_{i=1}^{I} g_i = \xi^*$, the optimal solution value of Problem \eqref{centralized1_capacity}.
\end{thm}
\begin{proof}
     By complementary slackness between Problems \eqref{centralized2_capacity} and \eqref{centralizeddual_capacity}, for all $i \in [I], j \in [J]$, we have
    $$
        {\everymath{\displaystyle}
                \scalebox{0.9}{
                $
                \begin{array}{l}
                    \vspace{10pt}
                    C_{i} \alpha_{j}^{*} V_{ij}^{*} = (\theta_{j}^{*} - \beta_{ij}^{*} - \gamma_{ij}^{*}) V_{ij}^{*}, \\
                    \vspace{10pt}
                    C_{i} \bar{\alpha}_{k}^{*} Z_{ik}^{*} = (\theta_{k}^{*} + \beta_{ik}^{*} - \gamma_{ik}^{*}) Z_{ik}^{*}, \\
                    \vspace{10pt}
                    C_{i} \sum_{j=1}^{J} \alpha_{j}^{*} V_{ij}^{*} = \sum_{j=1}^{J} (\theta_{j}^{*} - \beta_{ij}^{*} - \gamma_{ij}^{*}) V_{ij}^{*}, \\
                    C_{i} \sum_{k=1}^{I} \bar{\alpha}_{k}^{*} Z_{ik}^{*} = \sum_{k=1}^{I} (\theta_{k}^{*} + \beta_{ik}^{*} - \gamma_{ik}^{*}) Z_{ik}^{*}.
                \end{array}
                $
                }
        }
    $$
    Also, using equations \eqref{dual_norm_beta} and \eqref{dual_norm_gamma} and the constraints $\boldsymbol{\psi}_{i} = \boldsymbol{V}_{i}$, $\boldsymbol{\phi}_{i} = \boldsymbol{V}_{i}$, $\bar{\boldsymbol{\psi}}_{i} = x_{i}\boldsymbol{e}_{i} - \boldsymbol{Z}_{i}$ and $\bar{\boldsymbol{\phi}}_{i} = \boldsymbol{Z}_{i}$ of Problem \eqref{centralized2_capacity}, for all $i$,
    $$
        {\everymath{\displaystyle}
                \scalebox{0.9}{
                $
                \begin{array}{l}
                    \vspace{10pt}
                    (\boldsymbol{\beta}_{i}^{*})^{T} \boldsymbol{V}_{i}^{*} = (\boldsymbol{\beta}_{i}^{*})^{T} \boldsymbol{\psi}_{i}^{*} = \sigma_{i}^{*} \Gamma_{q} \| \boldsymbol{\psi}_{i}^{*}  \|_{\ell^{*}}, \\
                    \vspace{10pt}
                    (\boldsymbol{\gamma}_{i}^{*})^{T} \boldsymbol{V}_{i} = (\boldsymbol{\gamma}_{i}^{*})^{T} \boldsymbol{\phi}_{i}^{*} = \zeta_{i}^{*} \Gamma_{q} \| \boldsymbol{\phi}_{i}^{*}  \|_{\ell^{*}}, \\
                    \vspace{10pt}
                    (\bar{\boldsymbol{\beta}}_{i}^{*})^{T} (x_{i}^{*} \boldsymbol{e}_{i} - \boldsymbol{Z}_{i}^{*}) = (\bar{\boldsymbol{\beta}}_{i}^{*})^{T} \bar{\boldsymbol{\psi}}_{i}^{*} = \sigma_{i}^{*} \Delta_{p} \| \bar{\boldsymbol{\psi}}_{i}^{*}  \|_{\ell^{*}}, \\
                    (\bar{\boldsymbol{\gamma}}_{i}^{*})^{T} \boldsymbol{Z}_{i} = (\bar{\boldsymbol{\gamma}}_{i}^{*})^{T} \bar{\boldsymbol{\phi}}_{i}^{*} = \zeta_{i}^{*} \Delta_{p} \| \bar{\boldsymbol{\phi}}_{i}^{*}  \|_{\ell^{*}}.
                \end{array}
                $
                }
        }
    $$
  
    So, for all $i \in [I]$, using the same complementary slackness conditions as Theorem 1,
   $$ 
       {\everymath{\displaystyle}
       \scalebox{0.9}{
       $
            \begin{array}{l}
                F_{i} x_{i}^{*} + C_{i} u_{i}^{*} + C_{i} \sum_{j=1}^{J} \alpha_{j}^{*} V_{ij}^{*} + C_{i} \sum_{k=1}^{I} \bar{\alpha}_{k}^{*} Z_{ik}^{*} \\
                =  \rho_{i}x_{i}^{*} + \sigma_{i}^{*} p^{\max}_{i} x_{i}^{*} - \bar{\beta}_{ii}^{*} x_{i}^{*} + \mu^{*} u_{i}^{*} - \sigma_{i}^{*} u_{i}^{*} + \zeta_{i}^{*} u_{i}^{*} \\
                + \sum_{j=1}^{J} (\theta_{j}^{*} - \beta_{ij}^{*} - \gamma_{ij}^{*}) V_{ij}^{*} \\
                + \sum_{k=1}^{I} (\bar{\theta}_{k}^{*} + \bar{\beta}_{ik}^{*} - \bar{\gamma}_{ik}^{*}) Z_{ik}^{*} \\
                = \rho_{i}^{*} x_{i}^{*} + \mu^{*} u_{i}^{*} + \sum_{j=1}^{J} \theta_{j}^{*} V_{ij}^{*} \\
                + (\zeta_{i}^{*} u_{i}^{*} - \zeta_{i}^{*} \Gamma_{q} \| \boldsymbol{\phi}_{i}^{*}  \|_{\ell^{*}} - \zeta_{i}^{*} \Delta_{p} \| \bar{\boldsymbol{\phi}}_{i}^{*}  \|_{\ell^{*}}) \\
                + (\sigma_{i}^{*} p^{\max}_{i} x_{i}^{*} - \sigma_{i}^{*} u_{i}^{*} - \sigma_{i}^{*} \Gamma_{q} \| \boldsymbol{\psi}_{i}^{*}  \|_{\ell^{*}} - \sigma_{i}^{*} \Delta_{p} \| \bar{\boldsymbol{\psi}}_{i}^{*}  \|_{\ell^{*}}) \\
                = \rho_{i}^{*} x_{i}^{*} + \mu^{*} u_{i}^{*} + \sum_{j=1}^{J} \theta_{j}^{*} V_{ij}^{*} + \sum_{k=1}^{I} \bar{\theta}_{k}^{*} Z_{ik}^{*} = g_{i}.
            \end{array}
        $
        }
        }
    $$
    Next, we show that
    $$
        {\everymath{\displaystyle}
        \scalebox{0.9}{
        $
        \begin{array}{l}
            F_{i} x_{i}^{*} + C_{i} u_{i}^{*} + C_{i} \sum_{j=1}^{J} \alpha_{j}^{*} V_{ij}^{*} + C_{i} \sum_{k=1}^{I} \bar{\alpha}_{k}^{*} Z_{ik}^{*} \\
            = F_{i} x_{i}^{*} + C_{i} u_{i}^{*} + C_{i} \sum_{j=1}^{J} V_{ij}^{*} d_{j}^{*} + C_{i} \sum_{k=1}^{I} Z_{ik}^{*} r_{k}^{*} = f_{i}.
        \end{array}
        $
        }
        }
    $$
    Using equation \eqref{dual_norm_alpha} and the constraints $\boldsymbol{\omega} = \sum_{i=1}^{I} C_{i} \boldsymbol{V}_{i}$ and $\bar{\boldsymbol{\omega}} = \sum_{i=1}^{I} C_{i} \boldsymbol{Z}_{i}$ of Problem \eqref{centralized2_capacity} and $\nu=1$ of Problem \eqref{centralizeddual_capacity},
    $$
        {\everymath{\displaystyle}
            \scalebox{0.9}{
            $
            \begin{array}{l}
                 \sum_{i=1}^{I} C_{i} \sum_{j=1}^{J} \alpha_{j}^{*} V_{ij}^{*} = (\boldsymbol{\alpha}^{*})^{T} \boldsymbol{\omega}^{*} \\
                 = \nu^{*} \Gamma_{q} \| \boldsymbol{\omega}^{*} \|_{\ell^{*}} = \Gamma_{q} \|\boldsymbol{\omega}^{*}  \|_{\ell^{*}} \\
                 = \Gamma_{q} \|\sum_{i=1}^{I} C_{i} \boldsymbol{V}_{i}^{*} \|_{\ell^{*}} = \max_{\boldsymbol{d} \in \mathcal{D}} \sum_{i=1}^{I} C_{i} \sum_{j=1}^{J} V_{ij}^{*} d_{j} \\
                 = \sum_{i=1}^{I} C_{i} \sum_{j=1}^{J} V_{ij}^{*} d_{j}^{*},     
            \end{array}
            $
            }
        }
    $$
    and
    $$
        {\everymath{\displaystyle}
            \scalebox{0.9}{
            $
            \begin{array}{l}
                 \sum_{i=1}^{I} C_{i} \sum_{k=1}^{I} \bar{\alpha_{k}}^{*} Z_{ik}^{*} = (\bar{\boldsymbol{\alpha}}^{*})^{T} \bar{\boldsymbol{\omega}}^{*} \\
                 = \nu^{*} \Delta_{p} \| \bar{\boldsymbol{\omega}}^{*} \|_{\ell^{*}} = \Delta_{p} \|\bar{\boldsymbol{\omega}}^{*}  \|_{\ell^{*}} \\
                 = \Delta_{p} \|\sum_{i=1}^{I} C_{i} \boldsymbol{Z}_{i}^{*} \|_{\ell^{*}} = \max_{\boldsymbol{r} \in \mathcal{U}} \sum_{i=1}^{I} C_{i} \sum_{k=1}^{I} Z_{ik}^{*} r_{k} \\
                 = \sum_{i=1}^{I} C_{i} \sum_{k=1}^{I} Z_{ik}^{*} r_{k}^{*}. 
            \end{array}
            $
            }
            }
    $$
    
Then, $\boldsymbol{\alpha}$ is the worst-case load uncertainty, because $\boldsymbol{\alpha}^{*} = \arg \max_{\boldsymbol{d} \in \mathcal{D}}  \sum_{i=1}^{I} C_{i} \sum_{j=1}^{J} V_{ij}^* d_{j}$ and $\boldsymbol{\alpha}^{*}$ is in the uncertainty set $\mathcal{D}$ or $\| \boldsymbol{\alpha}^{*} \|_{\ell} \leq \Gamma_{q}$ by the constraints of Problem \eqref{centralizeddual_capacity}. Similarly, $\bar{\boldsymbol{\alpha}}$ is the worst-case capacity uncertainty, because $\bar{\boldsymbol{\alpha}}^{*} = \arg \max_{\boldsymbol{r} \in \mathcal{U}} \sum_{i=1}^{I} C_{i} \sum_{k=1}^{I} Z_{ik}^{*} r_{k}$. Note, $\boldsymbol{d}^{*}$ and $\boldsymbol{r}^{*}$ may not be unique.

So, $f_{i} = g_{i}$ for all $i$ and $\sum_{i=1}^{I} f_i=\sum_{i=1}^{I} g_i = \xi^*$. \end{proof}

\subsection{Absence of Self-Scheduling}

In this section, we show that when the worst-case uncertainty is realized, the generators are indifferent between the market schedule and their optimal dispatch, so they do not have an incentive to self-schedule. The ISO solves the \say{centralized} Problem \eqref{centralized2_capacity} and sets prices $\mu$ for the non-adaptive dispatch, $\boldsymbol{\theta}$ and $\bar{\boldsymbol{\theta}}$ for the adaptive dispatch, and $\boldsymbol{\rho}$ for the commitment. Also, each generator $i$ has commitment costs $F_{i}$ and dispatch costs $C_{i}$. The cost of the adaptive dispatch is a linear function of $\boldsymbol{d}^{*}$ and $\boldsymbol{r}^{*}$, as defined earlier and specified by the ISO. So, generator $i$ has revenue $\mu u_{i} + \sum_{j=1}^{J} \theta_{j} V_{ij} + \sum_{k=1}^{I} \bar{\theta}_{k} Z_{ik} + \rho_{i} x_{i}$, while it has a cost $F_{i} x_{i} + C_{i} u_{i} + C_{i} \sum_{j=1}^{J} V_{ij} d_{j}^{*} + C_{i} \sum_{k=1}^{I} Z_{ik} r_{k}^{*}$. Each generator $i$ decides if they will self-schedule by maximizing the difference between the revenue and the costs. They solve the following \say{decentralized} problem which maximizes their individual profit:
\begin{equation}
    \label{decentralized_capacity}
    {\everymath{\displaystyle}
    \scalebox{0.95}{
    $
        \begin{array}{rl}
            \max_{\boldsymbol{x}, \boldsymbol{u}, \boldsymbol{V}, \boldsymbol{Z}} & \mu u_{i} + \sum_{j=1}^{J} \theta_{j} V_{ij} + \sum_{k=1}^{I} \bar{\theta}_{k} Z_{ik} + \rho_{i} x_{i} \\
            & - (F_{i} x_{i} + C_{i} u_{i} + C_{i} \sum_{j=1}^{J} V_{ij} d_{j}^{*} + C_{i} \sum_{k=1}^{I} Z_{ik} r_{k}^{*}) \\
            \text{s.t.} & u_{i} + \Gamma_{q} \|\boldsymbol{V}_{i} \|_{\ell^{*}} + \Delta_{p} \| x_{i} \boldsymbol{e}_{i} - \boldsymbol{Z}_{i} \|_{\ell^{*}} \leq p^{\max}_{i} x_{i}, \\
            & u_{i} - \Gamma_{q} \|\boldsymbol{V}_{i} \|_{\ell^{*}} - \Delta_{p} \|\boldsymbol{Z}_{i} \|_{\ell^{*}}  \geq 0, \\
            & x_{i} \in \{0, 1\}, \\
        \end{array}
    $
    }
    }
\end{equation}
The constraints are the robust counterparts of the capacity and non-negativity constraints for each $i$, which means that the dispatch of the decentralized problem will be non-negative and less than the maximum capacity of generator $i$ for all $\boldsymbol{d} \in \mathcal{D}$ and $\boldsymbol{r} \in \mathcal{U}$.


Let $h_{i}(x_{i}, u_{i}, \boldsymbol{V}_{i}, \boldsymbol{Z}_{i})$ be the objective of the decentralized problem 
\eqref{decentralized_capacity} for commitment $x_{i}$ and dispatch $u_{i}, \boldsymbol{V}_{i}, \boldsymbol{Z}_{i}$. Also, let $\boldsymbol{x}^{*}, \boldsymbol{u}^{*}, \boldsymbol{V}^{*}, \boldsymbol{Z}^{*}$ be an optimal solution to the centralized problem \eqref{centralized2_capacity} and $x_{i}^{**}, u_{i}^{**}, \boldsymbol{V}_{i}^{**}, \boldsymbol{Z}_{i}^{**}$ be a solution to the decentralized problem \eqref{decentralized_capacity} for generator $i$. Generator $i$ has no incentive to self-schedule only if
$$    {\everymath{\displaystyle}
        h_{i}(x_{i}^{**}, u_{i}^{**}, \boldsymbol{V}_{i}^{**}, \boldsymbol{Z}_{i}^{**}) \leq h_{i}(x_{i}^{*}, u_{i}^{*}, \boldsymbol{V}_{i}^{*}, \boldsymbol{Z}_{i}^{*}).
    }
$$

\begin{thm}
    Suppose generator $i$ solves the decentralized problem \eqref{decentralized_capacity} using the prices of the centralized problem \eqref{centralized2_capacity}. Then, they cannot obtain a schedule giving greater profit than the centralized market schedule or
    $$    {\everymath{\displaystyle}
        h_{i}(x_{i}^{**}, u_{i}^{**}, \boldsymbol{V}_{i}^{**}, \boldsymbol{Z}_{i}^{**}) \leq h_{i}(x_{i}^{*}, u_{i}^{*}, \boldsymbol{V}_{i}^{*}, \boldsymbol{Z}_{i}^{*}).
    }
$$
\end{thm}

\begin{proof}

Let $\mu^{*}, \boldsymbol{\theta}^{*}, \boldsymbol{\rho}^{*}$ be the  prices determined by  the dual variables of the Problem \eqref{centralizeddual_capacity}. Using the results of Theorem \ref{thm:thm2_capacity}, $h_{i}(x_{i}^{*}, u_{i}^{*}, \boldsymbol{V}_{i}^{*}) = 0$, because $f_{i} = g_{i}$.

Consider $h_{i}(x_{i}^{**}, u_{i}^{**}, \boldsymbol{V}_{i}^{**})$, 
$$ 
{\everymath{\displaystyle}
    \scalebox{0.9}{
    $
        \begin{array}{llr}
            & h_{i}(x_{i}^{**}, u_{i}^{**}, \boldsymbol{V}_{i}^{**}) = \mu^{*} u_{i}^{**} + \sum_{j=1}^{J} \theta_{j}^{*} V_{ij}^{**} + \sum_{k=1}^{I} \bar{\theta}_{k}^{*} Z_{ik}^{**} + \rho_{i}^{*} x_{i}^{**} \\
            & - (F_{i} x_{i}^{**} + C_{i} u_{i}^{**}  + C_{i} \sum_{j=1}^{J} V_{ij}^{**}  d_{j}^{*} + C_{i} \sum_{k=1}^{I} Z_{ik}^{**}  r_{k}^{*}) \\
            & \leq (\mu^{*} - C_{i}) u_{i}^{**} + (\rho_{i}^{*} - F_{i}) x_{i}^{**} \\
            \vspace{5pt}
            & + \sum_{j=1}^{J} (\theta_{j}^{*} - C_{i} \alpha_{j}^{*}) V_{ij}^{**} - \sum_{j=1}^{J} (\bar{\theta}_{k}^{*} - C_{i} \bar{\alpha}_{k}^{*}) Z_{ik}^{**}  \\
            \vspace{5pt}
            & + \sigma_{i}^{*} (p^{\max}_{i} x_{i}^{**} - u_{i}^{**} - \Gamma_{q} \|V_{i}^{**} \|_{\ell^{*}} - \Delta_{p} \| x_{i}^{**} \boldsymbol{e}_{i} - \boldsymbol{Z}_{i}^{**} \|_{\ell^{*}}) \\
            \vspace{5pt}
            & + \zeta_{i}^{*} (u_{i}^{**} - \Gamma_{q} \|V_{i}^{**} \|_{\ell^{*}} - \Delta_{p} \|\boldsymbol{Z}_{i} \|_{\ell^{*}})
        \end{array}
    $
    }
    }
$$

$$ 
{\everymath{\displaystyle}
    \scalebox{0.9}{
    $
        \begin{array}{llr}
            & \leq (\mu^{*} - \sigma_{i}^{*} + \zeta_{i}^{*} - C_{i}) u_{i}^{**} + (\rho_{i}^{*} + \sigma_{i}^{*} p^{\max}_{i} - \beta_{ii}^{*} - F_{i}) x_{i}^{**} \\
            & + \sum_{j=1}^{J} (\theta_{j}^{*} - \beta_{ij}^{*} - \gamma_{ij}^{*} - C_{i} \alpha_{j}^{*}) V_{ij}^{**} \\
            & + \sum_{k=1}^{I} (\theta_{k}^{*} + \beta_{ik}^{*} - \gamma_{ik}^{*} - C_{i} \alpha_{k}^{*}) Z_{ik}^{**} \leq 0. \\
        \end{array}
    $
    }
    }
$$

The first inequality is valid because $p^{\max}_{i} x_{i}^{**} - u_{i}^{**} - \Gamma_{q} \|V_{i}^{**} \|_{\ell^{*}} - \Delta_{p} \| x_{i}^{**} \boldsymbol{e}_{i} - \boldsymbol{Z}_{i}^{**} \|_{\ell^{*}} \geq 0$, $u_{i}^{**} - \Gamma_{q} \|V_{i}^{**} \|_{\ell^{*}} - \Delta_{p} \|\boldsymbol{Z}_{i}^{**} \|_{\ell^{*}} \geq 0$ by the constraints of Problem \eqref{decentralized_capacity} and $\sigma_{i}^{*} \geq 0$ and $\zeta_{i}^{*} \geq 0$ by Problem \eqref{centralizeddual_capacity}. The third inequality is valid because $(\mu^{*} - \sigma_{i}^{*} + \zeta_{i}^{*} - C_{i}) = 0$, $(\theta_{j}^{*} - \beta_{ij}^{*} - \gamma_{ij}^{*} - C_{i} \alpha_{j}^{*})=0$, $(\theta_{k}^{*} + \beta_{ik}^{*} - \gamma_{ik}^{*} - C_{i} \alpha_{k}^{*}) = 0$ and $(\rho_{i}^{*} + \sigma_{i}^{*} p^{\max}_{i} - \beta_{ii}^{*} - F_{i}) \leq 0$ by the constraints of Problem \eqref{centralizeddual_capacity}. Also, the second inequality is valid, because
\begin{equation*}
    \scalebox{0.85}{
    $
    (\boldsymbol{\beta}_{i}^{*})^{T} \boldsymbol{V}_{i}^{**} \leq \| \boldsymbol{\beta}_{i}^{*} \|_{\ell} \; \; \| \boldsymbol{V}_{i}^{**} \|_{\ell^{*}} \leq \sigma_{i}^{*} \Gamma_{q} \| \boldsymbol{V}_{i}^{**} \|_{\ell^{*}},
    $
    }
\end{equation*}
\begin{equation*}
    \scalebox{0.85}{
    $
    (\boldsymbol{\gamma}_{i}^{*})^{T} \boldsymbol{V}_{i}^{**} \leq \| \boldsymbol{\gamma}_{i}^{*} \|_{\ell} \; \; \|\boldsymbol{V}_{i}^{**} \|_{\ell^{*}} \leq \zeta_{i}^{*} \Gamma_{q} \|\boldsymbol{V}_{i}^{**} \|_{\ell^{*}},
    $
    }
\end{equation*}
\begin{equation*}
    \scalebox{0.85}{
    $
    (\bar{\boldsymbol{\beta}}_{i}^{*})^{T} (x_{i}^{**} \boldsymbol{e}_{i} - \boldsymbol{Z}_{i}^{**}) \leq \| \bar{\boldsymbol{\beta}}_{i}^{*} \|_{\ell} \; \; \| x_{i}^{**} \boldsymbol{e}_{i} - \boldsymbol{Z}_{i}^{**} \|_{\ell^{*}} \leq \sigma_{i}^{*} \Delta_{p} \|x_{i}^{**} \boldsymbol{e}_{i} - \boldsymbol{Z}_{i}^{**} \|_{\ell^{*}},
    $
    }
\end{equation*}
\begin{equation*}
    \scalebox{0.85}{
    $
    (\bar{\boldsymbol{\gamma}}_{i}^{*})^{T} \boldsymbol{Z}_{i}^{**} \leq \| \bar{\boldsymbol{\gamma}}_{i}^{*} \|_{\ell} \; \; \|\boldsymbol{Z}_{i}^{**} \|_{\ell^{*}} \leq \zeta_{i}^{*} \Delta_{p} \|\boldsymbol{Z}_{i}^{**} \|_{\ell^{*}},
    $
    }
\end{equation*}

by the dual norm properties and the constraints of Problem \eqref{centralizeddual_capacity}. Specifically, if $\boldsymbol{x}, \boldsymbol{z}$ are two $n$-dimensional vectors, then $\| \boldsymbol{z}\|_{\ell^{*}} = \{ \max_{\boldsymbol{x}} \boldsymbol{z}^{T} \boldsymbol{x}:  \| \boldsymbol{x}\|_{\ell} \leq 1 \}$ and $ \boldsymbol{z}^{T} \boldsymbol{x} \leq \| \boldsymbol{x}\|_{\ell} \| \boldsymbol{z}\|_{\ell^{*}}$ for all $\boldsymbol{x}, \boldsymbol{z}$ \cite{bertsekas1997nonlinear}.


So, $h_{i}(x_{i}^{**}, u_{i}^{**}, \boldsymbol{V}_{i}^{**}, \boldsymbol{Z}_{i}^{**}) \leq h_{i}(x_{i}^{*}, u_{i}^{*}, \boldsymbol{V}_{i}^{*}, \boldsymbol{Z}_{i}^{*})$, i.e., no generator $i$ can obtain a solution giving a greater profit than the centralized solution. \end{proof}

Also, $\boldsymbol{x}^{*}, \boldsymbol{u}^{*}, \boldsymbol{V}^{*}, \boldsymbol{Z}^{*}$ are market-clearing or $\sum_{i=1}^{I} (u_{i} + \sum_{j=1}^{J} V_{ij} {d}_{j} + \sum_{k=1}^{I} Z_{ik} r_{k}) \geq \sum_{j=1}^{J} (d_{j} + \bar{q}_{j})$ for all $\boldsymbol{d} \in \mathcal{D}$ and $\boldsymbol{r} \in \mathcal{U}$, so the generators do not have incentive to self-schedule.

\section{Example with load uncertainty}
\label{sec:sec_load_example}

In this section, we demonstrate adaptive pricing in detail on the Scarf example with load uncertainty. The day-ahead payments and prices in the ARO problem are larger than those in the deterministic problem, because the commitments and dispatch are more conservative. However, the ARO approach reduces the need for ad hoc corrections, so the total cost may be significantly lower once the uncertainty is realized.

The Scarf example, after fixing the binary variables to their optimal values, is based on the following formulation
$$
    {\everymath{\displaystyle}
    \scalebox{0.9}{
    $
    \begin{array}{rlr}
    \min_{\boldsymbol{x}, \boldsymbol{p}} & \sum_{i=1}^{I} F_{i} x_{i} + C_{i} p_{i} \\
    \text{s.t.} & \sum_{i=1}^{I} p_{i} = \sum_{j=1}^{J} \bar{q}_{j}, \\
    & p_{i} \leq p^{\max}_{i} x_{i}, \quad \forall i \in [I], \\
    & x_{i} = x_{i}^{*}, \quad \forall i \in [I], \\
    & x_{i}, p_{i} \geq 0, \quad \forall i \in [I]. \\
    \end{array}
    $
    }
    }
$$
There are only two types of generators, whose costs and capacities are asymmetrical. Their characteristics are summarized in Table \ref{tab:tab1}.

\begin{table}[ht]
    \centering
    \caption{The two types of generators in the Scarf example and their commitment costs, dispatch costs and maximum capacity.}
    \begin{tabular}{c|ccc}
         \textbf{Type} & $F_{i}$ & $C_{i}$ & $p_{i}^{\max}$ \\
         \hline
         1 & 53 & 3 & 16 \\
         2 & 30 & 2 & 7
    \end{tabular}
    \label{tab:tab1}
\end{table}

Suppose we have two generators of Type 1 and 6 generators of type 2. We also have five consumers with expected load $[8, 8, 3, 5, 16]$, so the total expected load is $\sum_{j=1}^{J} \bar{q}_{j} = 40$.

The results for the deterministic problem are summarized in Table \ref{tab:tab2}.

\begin{itemize}
    \item Objective: \$ 260,
    \item Dual price of the load: 2 \$ / unit.
\end{itemize}

\begin{table}[ht]
    \centering
    \caption{Commitment and dispatch in the deterministic Scarf example. We choose to turn on all generators of type 2, because they meet the expected demand and have the least cost.}
    \begin{tabular}{c|cccccccc}
    \textbf{Type} & 1 & 1 & 2 & 2 & 2 & 2 & 2 & 2 \\
     \hline
     \textbf{Commitment} & 0 & 0 & 1 & 1 & 1 & 1 & 1 & 1 \\
     \textbf{Dispatch} & 0 & 0 & 5 & 7 & 7 & 7 & 7 & 7
    \end{tabular}
    \label{tab:tab2}
\end{table}

Table \ref{tab:tab3} summarizes the pay-as-bid payments, which are $F_{i} x_{i}^{*} + C_{i} p_{i}^{*}$ for each $i$.

\begin{table}[ht]
    \centering
    \caption{Pay-as-bid payments in the Scarf example. The bids consist of commitment and dispatch costs.}
    \begin{tabular}{c|cccccccc}
     \textbf{Type} & 1 & 1 & 2 & 2 & 2 & 2 & 2 & 2 \\
     \hline
     $F_{i} x_{i}^{*}$ & 0 & 0 & 30 & 30 & 30 & 30 & 30 & 30 \\
     $C_{i} p_{i}^{*}$ & 0 & 0 & 10 & 14 & 14 & 14 & 14 & 14 \\
     \textbf{Payments} & 0 & 0 & 40 & 44 & 44 & 44 & 44 & 44
    \end{tabular}
    \label{tab:tab3}
\end{table}

The uniform price payments are $2 p_{i}^{*} + \rho_{i}^{*} x_{i}^{*}$ for each $i$, where $\rho_{i}^{*}$ is the dual variable of the $x_{i} = x_{i}^{*}$ constraint. They are are summarized in Table \ref{tab:tab4}. In general, $F_{i} x_{i}^{*}$ is not the same as $\rho_{i}^{*} x_{i}^{*}$.

\begin{table}[ht]
    \centering
    \caption{Uniform price payments with uplifts in the Scarf example.}
    \begin{tabular}{c|cccccccc}
     \textbf{Type} & 1 & 1 & 2 & 2 & 2 & 2 & 2 & 2 \\
     \hline
     $2 p_{i}^{*}$ & 0 & 0 & 10 & 14 & 14 & 14 & 14 & 14 \\
     Uplifts & 0 & 0 & 30 & 30 & 30 & 30 & 30 & 30 \\
     \textbf{Payments} & 0 & 0 & 40 & 44 & 44 & 44 & 44 & 44
    \end{tabular}
    \label{tab:tab4}
\end{table}



In the adaptive problem we also need to select the budget of uncertainty. Suppose we use a budget uncertainty set with $\Gamma_{q} = 20$ , which means that we are protected from an increase of the expected demand from 40 to 60 or a decrease from 40 to 20. In this case, $\Delta_{p}=0$. The results for the ARO problem are summarized in Table \ref{tab:tab5}.

\begin{itemize}
    \item Objective: \$ 378,
    \item Dual price of the load: 3 \$ / unit.
\end{itemize}

\begin{table}[ht]
    \centering
    \caption{Commitment and non-adaptive dispatch in the Scarf example. Generator 3, which is of Type 2, will have only deterministic dispatch, which does not depend on the uncertain load.}
    \begin{tabular}{c|cccccccc}
     \textbf{Type} & 1 & 1 & 2 & 2 & 2 & 2 & 2 & 2 \\
     \hline
     \textbf{Commitment} & 1 & 1 & 1 & 0 & 0 & 1 & 1 & 1 \\
     $u_{i}^{*}$ & 8 & 14.5 & 7 & 0 & 0 & 3.5 & 3.5 & 3.5 
    \end{tabular}
    \label{tab:tab5}
\end{table}

The LDR for each generator $i$ is $p_{i}(\boldsymbol{d}) = u_{i} + \sum_{j=1}^{J} V_{ij} d_{j}$. The matrix $\boldsymbol{V}^{*}$ is

$$
\scalebox{0.8}{
$
    \begin{bmatrix}
    0.4 & 0.4 & 0.4 & 0.4 & 0.4 \\
    0.075 & 0.075 & 0.075 & 0.075 & 0.075 \\
    0.0 & 0.0 & 0.0 & 0.0 & 0.0 \\
    0.0 & 0.0 & 0.0 & 0.0 & 0.0 \\
    0.0 & 0.0 & 0.0 & 0.0 & 0.0 \\
    0.175 & 0.175 & 0.175 & 0.175 & 0.175 \\
    0.175 & 0.175 & 0.175 & 0.175 & 0.175 \\
    0.175 & 0.175 & 0.175 & 0.175 & 0.175 \\
    \end{bmatrix}. 
$
}
$$
For each $i$, $V_{ij}^{*}$ is the same for all consumers $j$, because the load $\sum_{j=1}^{J} \bar{q}_{j}$ bundles consumers together. This can change if we add different coefficients for each $j$ in the uncertainty set.

Table \ref{tab:tab6} summarizes the day-ahead payments in an adaptive pay-as-bid scheme, where each generator $i$ is paid $F_{i} x_{i}^{*} + C_{i} u_{i}^{*}$ and the total payments are \$ 328.5.

\begin{table}[ht]
    \centering
    \caption{Pay-as-bid payments in the adaptive Scarf example. The bids consist of commitment and non-adaptive dispatch costs.}
    \begin{tabular}{c|cccccccc}
     \textbf{Type} & 1 & 1 & 2 & 2 & 2 & 2 & 2 & 2 \\
     \hline
     $F_{i} x_{i}^{*}$ & 53 & 53 & 30 & 0 & 0 & 30 & 30 & 30 \\
     $C_{i}u_{i}^{*}$ & 24 & 43.5 & 14 & 0 & 0 & 7 & 7 & 7 \\
     \textbf{Payments} & 77 & 96.5 & 44 & 0 & 0 & 37 & 37 & 37
    \end{tabular}
    \label{tab:tab6}
\end{table}




Table \ref{tab:tab7} summarizes the day-ahead payments in the adaptive marginal pricing scheme. They payments are the same as the pay-as-bid scheme. Note that the uplifts all less than the deterministic case for each generator of Type 2.

\begin{table}[ht]
    \centering
    \caption{Marginal price payments in the adaptive Scarf example.}
    \begin{tabular}{c|cccccccc}
     \textbf{Type} & 1 & 1 & 2 & 2 & 2 & 2 & 2 & 2 \\
     \hline
     $\mu^{*} u_{i}^{*}$ & 24 & 43.5 & 21 & 0 & 0 & 10.5 & 10.5 & 10.5 \\
     Uplifts & 53 & 53 & 23 & 0 & 0 & 26.5 & 26.5 & 26.5 \\
     \textbf{Payments} & 77 & 96.5 & 44 & 0 & 0 & 37 & 37 & 37
    \end{tabular}
    \label{tab:tab7}
\end{table}

Both the deterministic problem and the commitments and non-adaptive dispatch in the ARO problem try to meet a load of 40. However, the deterministic cost is \$ 280, while the non-adaptive ARO cost is \$ 328.5. This is because the ARO commitment is also feasible for an increase in the load up to 20, which is the worst-case scenario in our uncertainty set. The deterministic problem cannot meet such an increase, because Type 2 generators are used at capacity, except the first one, which can provide only two more units of power. To meet the demand of $40+20$, we would need to turn on more generators the following day, which would be very costly and would require large payments.



\subsection{Intra-day dispatch and payments}

In this section, we consider the intra-day economic dispatch problem and the pricing implications of uncertain scenarios.

Suppose that the following day the uncertainty $\boldsymbol{d} \geq 0$ is realized, so the realized load is $\boldsymbol{q} = \bar{\boldsymbol{q}} + \boldsymbol{d}$. We have made some commitments in the day-ahead market, so we can solve the following linear optimization problem (LP) to find the optimal dispatch, based on the new data.
$$
    {\everymath{\displaystyle}
    \scalebox{0.9}{
    $
        \begin{array}{rlr}
            \min_{\boldsymbol{p}} & \sum_{i=1}^{I} C_{i} p_{i} \\
            \text{s.t.} & \sum_{i=1}^{I} p_{i} = \sum_{j=1}^{J} d_{j}, \\
            & p_{i} \leq p^{\max}_{i} x_{i}^{*} - u_{i}^{*}, \quad \forall i \in [I], \\
            & p_{i} \geq 0, \quad \forall i \in [I], \\
        \end{array} 
    $
    }
    }
$$
where $x_{i}^{*}$ and $u_{i}^{*}$ are the solution for generator $i$ based on the ARO problem. So, we want to meet only the additional realized load $\boldsymbol{d}$, while we have already committed to produce $u_{i}^{*}$ to meet the expected load.

The intra-day electricity price is 2 for loads smaller than 10 and 3 for loads larger than 10, which is equal to the ARO price. Using ARO, we can provide the intra-day electricity price for all realizations of the uncertain parameters. As a result, the pricing is more transparent and predictable and the market participants can plan their bids.

\begin{figure}
    \begin{center}
        \includegraphics[width=0.4\textwidth]{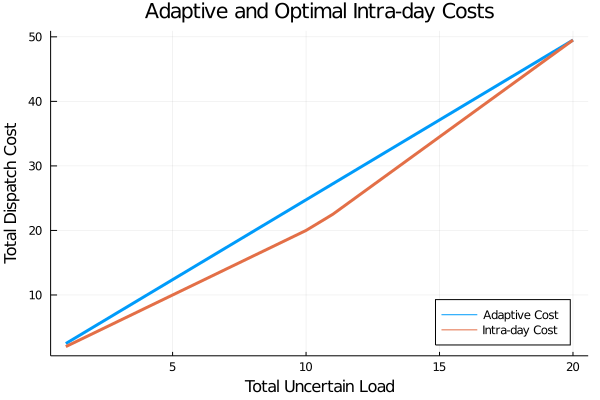}
        \caption{Comparison of the optimal intra-day dispatch cost and the adaptive part of the day-ahead cost. The adaptive part of the day-ahead cost is an upper bound and is equal to the optimal intra-day dispatch for the worst-case scenario.}
        \label{fig:fig1}
    \end{center}
\end{figure}

In Figure \ref{fig:fig1}, we plot the cost of the LP and compare it to the adaptive part of the ARO cost $\sum_{i=1}^{I} C_{i} \sum_{j=1}^{J} V_{ij}^{*} d_{j}$. We gradually increase the total realized load $\sum_{j=1}^{J} d_{j}$ from zero to $\Gamma_{q}=20$, which is the worst-case scenario. The adaptive dispatch is an upper bound to the optimized intra-day dispatch. However, we have optimized for the worst-case uncertainty, so the adaptive problem and the LP will have the same cost, if the worst-case load is realized.


\section{Example with load and capacity uncertainty}
\label{sec:sec_load_cap_example}

We use the example of Section \ref{sec:sec_load_example} with uncertain load and consider the uncertainty in the capacity of each generator. Again, we set $\Gamma_{q} = 20$, which means that we are protected from a load increase from 40 to 60. Also, we choose $\Delta_p = 0.5$, which means that we are protected from a decrease of the expected capacity from $\bar{p}_{i}^{\max}$ to $\bar{p}_{i}^{\max}-0.5$ for each generator $i$. The results for the ARO problem are summarized in Table \ref{tab:tab8}.
\begin{itemize}
    \item Objective: \$ 402.25,
    \item Dual price of the load: 3 \$ / unit.
\end{itemize}

\begin{table}[ht]
    \centering
    \caption{Commitment and non-adaptive dispatch in the adaptive Scarf example with load and capacity uncertainty.}
    \begin{tabular}{c|cccccccc}
         \textbf{Type} & 1 & 1 & 2 & 2 & 2 & 2 & 2 & 2 \\
         \hline
         \textbf{Commitments} & 1 & 1 & 1 & 1 & 1 & 1 & 1 & 0 \\
         $u_{i}^{*}$ & 8.0 & 8.0 & 3.5 & 3.5 & 3.5 & 6.75 & 6.75 & 0
    \end{tabular}
    \label{tab:tab8}
\end{table}

The LDR is $p_{i}(\boldsymbol{d}, \boldsymbol{r}) = u_{i} + \sum_{j=1}^{J} V_{ij} d_{j} + \sum_{k=1}^{I} Z_{ik} r_{k}$. The matrix $\boldsymbol{V}^{*}$ is
$$ 
\scalebox{0.8}{
$
\begin{bmatrix}
 0.25625 &  0.25625 &  0.25625 &  0.25625 &  0.25625 \\
 0.25625 &  0.25625 &  0.25625 &  0.25625 &  0.25625 \\
 0.1625 &   0.1625 &   0.1625 &   0.1625 &   0.1625 \\
 0.1625 &   0.1625 &   0.1625 &   0.1625 &   0.1625 \\
 0.0 &      0.0 &      0.0 &      0.0 &      0.0 \\
 0.0 &     0.0 &      0.0 &      0.0 &      0.0 \\
 0.1625 &   0.1625 &   0.1625 &   0.1625 &   0.1625 \\
 0.0 &      0.0 &      0.0 &      0.0 &      0.0
\end{bmatrix}.
$
}
$$

The matrix $\boldsymbol{Z}^{*}$ is
$$ 
\scalebox{0.8}{
$
\begin{bmatrix}
  5.75 &   4.75 &   5.75 &  -5.75 &  -5.75 &  -5.75 &   5.75 &  -5.75 \\ 
 -5.75 &  -4.75 &  -5.75 &   5.75 &   5.75 &   5.75 &  -5.75 &   5.75 \\ 
  0.0 &   -0.5 &    0.5 &   -0.5 &    0.0 &    0.5 &    0.0 &    0.5 \\ 
  0.5 &    0.0 &    0.0 &    0.5 &    0.5 &    0.0 &    0.5 &   -0.5 \\ 
 -0.5 &   -0.5 &   -0.5 &    0.0 &    0.5 &   -0.5 &   -0.5 &    0.5 \\ 
 -0.5 &    0.5 &    0.5 &   -0.5 &   -0.5 &    0.5 &   -0.5 &   -0.5 \\ 
  0.5 &   0.5 &   -0.5 &    0.5 &   -0.5 &   -0.5 &    0.5 &    0.0 \\
  0.0 &    0.0 &    0.0 &    0.0 &    0.0 &    0.0 &    0.0 &    0.0 
\end{bmatrix}.
$
}
$$

Note that the sum of each column of $\boldsymbol{V}^{*}$ is one and of $\boldsymbol{Z}^{*}$ is zero, because of the robust constraint $\Gamma_{q} \|\boldsymbol{1} - \sum_{i=1}^{I} \boldsymbol{V}_{i} \|_{\ell^{*}} + \Delta_{p} \| - \sum_{i=1}^{I} \boldsymbol{Z}_{i} \|_{\ell^{*}} \leq 0$. The norms are non-negative, so $\|\boldsymbol{1} - \sum_{i=1}^{I} \boldsymbol{V}_{i} \|_{\ell^{*}} = 0$ and $\| \sum_{i=1}^{I} \boldsymbol{Z}_{i} \|_{\ell^{*}} = 0$.

Table \ref{tab:tab9} summarizes the day-ahead payments in a pay-as-bid scheme, where the total payments are \$ 352.

\begin{table}[ht]
    \centering
    \caption{Pay-as-bid payments in the adaptive Scarf example with load and capacity uncertainty.}
    \begin{tabular}{c|cccccccc}
     \textbf{Type} & 1 & 1 & 2 & 2 & 2 & 2 & 2 & 2 \\
     \hline
     $F_{i} x_{i}^{*}$ & 53 & 53 & 30 & 30 & 30 & 30 & 30 & 0 \\
     $C_{i} u_{i}^{*}$ & 24 & 24 & 7 & 7 & 13.5 & 13.5 & 7 & 0 \\
     \textbf{Payments} & 77 & 77 & 37 & 37 & 43.5 & 43.5 & 37.0 & 0
    \end{tabular}
    \label{tab:tab9}
\end{table}





Table \ref{tab:tab10} summarizes the day-ahead payments in the adaptive marginal pricing scheme. The payments are the same as the pay-as-bid scheme. The uplifts for Type 2 generators are still smaller than the deterministic uplifts, while they are close to the adaptive uplifts with load uncertainty.

\begin{table}[ht]
    \centering
    \caption{Marginal price payments in the adaptive Scarf example with load and capacity uncertainty.}
    \begin{tabular}{c|cccccccc}
     \textbf{Type} & 1 & 1 & 2 & 2 & 2 & 2 & 2 & 2 \\
     \hline
     $\mu^{*} u_{i}^{*}$ & 24 & 24 & 10.5 & 10.5 & 20.25 & 20.25 & 10.5 & 0 \\
     Uplifts & 53 & 53 & 26.5 & 26.5 & 23.25 & 23.25 & 26.5 & 0 \\
     \textbf{Payments} & 77 & 77 & 37 & 37 & 43.5 & 43.5 & 37.0 & 0
    \end{tabular}
    \label{tab:tab10}
\end{table}

In this case, we also protect against drops in the available capacity, so we are more conservative. As a result, the prices and payments increase from \$ 328.5 to \$ 352. We turn on all generators, as the commitments in the deterministic problem and in the ARO problem with load capacity are not robust to a decrease of 0.5 in the capacity of the generators.


\section{Multiperiod Pricing}
\label{sec:sec_multi}

In this section, we present our method on a realistic multiperiod formulation with ramp constraints and compare it to deterministic marginal and convex hull pricing. Again, the ARO commitments and dispatch are more conservative, so there is an increase in the payments and the prices compared to the deterministic case. However, there are no uplifts related to optimality gaps that are present in convex hull pricing.

Consider the following example by Chen et al. \cite{chen2020unified}, which includes two generators and a 3-hour horizon with expected load $\bar{q}_{t}$ 95, 100, and 130 MW. G1 has a maximum of 100 MW, and energy offer $\$10$ /MWh. G2 has a 20 MW minimum, 35 MW maximum, energy offer $\$50$ /MWh, start-up cost $\$1000$, no-load cost $\$30$, ramp rate 5 MW/hour, start-up rate 22.5 MW/hour, shut-down rate 35 MW/hour, minimum up/down times of one hour, and is initially offline. The UC formulation is provided below.
$$
    {\everymath{\displaystyle}
    \scalebox{0.95}{
    $
        \begin{array}{rlr}
        \min_{\boldsymbol{x}, \boldsymbol{p}} & \sum_{t=1}^{3} 10 p_{1, t}+30 x_{2, t}^{ON} +50 p_{2, t}+1000 x_{2, t}^{RU} \\
        \text{s.t.}& p_{1, t}+p_{2, t}= \bar{q}_{t}, & 1 \leq t , \\
        & 0 \leq p_{1, t} \leq 100, & 1 \leq t , \\
        & 20 x_{2, t}^{ON} \leq p_{2, t} 5 x_{2, t}^{ON}, & 1 \leq t , \\
        & p_{2, t}-p_{2, t-1} \leq 5 x_{2, t-1}^{ON} +22.5 x_{2, t}^{RU}, & 1 \leq t , \\
        & p_{2, t-1}-p_{2, t} \leq 5 x_{2, t}^{ON} +35 x_{2, t}^{RD}, & 2 \leq t , \\
        & x_{2, t}^{ON} - x_{2, t-1}^{ON} = x_{2, t}^{RU} - x_{2, t}^{RD}, & 1 \leq t , \\
        & x_{2, t}^{RU} \leq x_{2, t}^{ON}, \quad x_{2, t}^{RU} \leq 1-x_{2, t-1}^{ON}, & 1 \leq t ,
        \end{array}
    $
    }
    }
$$
with $\boldsymbol{p} \geq 0$ and $x_{2, t}^{ON}, x_{2, t}^{RU}, x_{2, t}^{RD} \in \{0, 1\} \; \forall t$ representing the status, start-up and shut-down variables respectively. The results for the deterministic problem are summarized in Table \ref{tab:tab11}. We turn on G2 at $t=1$ and both generators are on for all time periods. The objective is \$ 7340 and the electricity price is $\boldsymbol{\mu} = [10, 10, 90]$  \$ / MW.

\begin{table}[ht]
    \centering
    \caption{Dispatch solution for each generator at each time period in the deterministic multiperiod problem.}
    \begin{tabular}{c|cccccccc}
     \textbf{Generator} & 1 & 2 \\
     \hline
     \textbf{Dispatch} & [75, 75, 100] & [20, 25, 30]
    \end{tabular}
    
    \label{tab:tab11}
\end{table}

The convex hull payments are summarized in Table \ref{tab:tab12}. The convex hull prices are $\boldsymbol{\mu}^{CH} = [10, 10, 276]$  \$ / MW, so G1 is paid $\$$ 2500, while G2 is paid $\$$ 4445 and an uplift of $\$$ 365 \cite{andrianesis2021computation}.

\begin{table}[ht]
    \centering
    \caption{Convex hull payments and uplifts for each generator at each time period in the multiperiod example.}
    \begin{tabular}{c|cccccccc}
     \textbf{Generator} & 1 & 2 \\
     \hline
     $\mu^{CH}_{t} \; p_{i, t}^{*}$ & [750, 750, 27600] & [200, 250, 8250] \\
     Uplifts & -26600 & -4255 \\
     \textbf{Payments} & 2500 & 4445
    \end{tabular}
    \label{tab:tab12}
\end{table}

The objective of the convex hull method is $\$$ 365 less than the optimal objective, because of the duality gap. So, there is an additional uplift of $\$$ 365 that is paid to G2.

The deterministic pay-as-bid payments are summarized in Table \ref{tab:tab13}. They are are the same as the convex hull payments but do not feature a duality gap. In addition, these payments are equivalent to a marginal pricing scheme by \cite{o2005efficient}.

\begin{table}[ht]
    \centering
    \caption{Pay-as-bid payments for each generator at each time period in the multiperiod example.}
    \begin{tabular}{c|cccccccc}
     \textbf{Generator} & 1 & 2 \\
     \hline
     Commitment Cost & [0, 0, 0] & [1030, 30, 30] \\
     Dispatch Cost & [750, 750, 1000] & [1000, 1250, 1500] \\
     Total Cost & [750, 750, 1000] & [2030, 1280, 1530] \\
     \textbf{Payments} & 2500 & 4840
    \end{tabular}
    \label{tab:tab13}
\end{table}

In the ARO problem we use budget uncertainty sets with $\Gamma_{q} = [10, 10, 2]$ and $\Delta_{p} = [0, 7.5, 0.5]$ at each time period. The objective increases to \$ 7860 and the electricity price is $\boldsymbol{\mu}^{*} = [10, 10, 130]$ \$ / MW. Again, we turn on G2 at $t=1$ and both generators are on for all time periods. 

\begin{table}[ht]
    \centering
    \caption{Non-adaptive dispatch solution for each generator at each time period in the ARO multiperiod problem.}
    \begin{tabular}{c|cccccccc}
     \textbf{Generator} & 1 & 2 \\
     \hline
     \textbf{Dispatch} & [72.5, 72.5, 97.5] & [22.5, 27.5, 32.5] \\
    \end{tabular}
    \label{tab:tab14}
\end{table}

\noindent The LDR is $p_{i, t}(\boldsymbol{d}, \boldsymbol{r}) = u_{it} + \sum_{j=1}^{J} V_{ijt} d_{jt} + \sum_{k=1}^{I} Z_{ikt} r_{kt}$. The non-adaptive dispatch results are summarized in Table \ref{tab:tab14}. The matrix $\boldsymbol{V}_{t}^{*}$ at each time period is
\begin{equation*}
\scalebox{0.8}{
$
    \begin{array}{cc}
    \vspace{8pt}
    \begin{bmatrix}
    1.0 & 1.0 & 1.0 & 1.0 & 1.0 \\
    0.0 & 0.0 & 0.0 & 0.0 & 0.0
    \end{bmatrix} &  t = 1, 2, 3
    \end{array}
$
}
\end{equation*}
with the rows corresponding to the G1 and G2 and the columns corresponding to three consumers that share the load. Also, $\boldsymbol{Z}^{*}$ contains only zeros.


The day-ahead payments to the generators are \$ 7640. They pay-as-bid payments are summarized in Table \ref{tab:tab15} and each generator is paid $F_{i}^{ON} x_{i, t}^{ON} + F_{i}^{RU} x_{i, t}^{RU} + F_{i}^{RD} x_{i, t}^{RD} + C_{i, t} u_{i, t}^{*}$, where the coefficients correspond to the commitment and dispatch costs.

\begin{table}[ht]
    \centering
    \caption{Pay-as-bid payments for each generator at each time period in the ARO multiperiod example.}
    \begin{tabular}{c|cc}
     \textbf{Generator} & 1 & 2 \\
     \hline
     Commitment Cost & [0, 0, 0] & [1030, 30, 30] \\
     Dispatch Cost & [725, 725, 975] & [1125, 1375, 1625] \\
     Total Cost & [725, 725, 975] & [2155, 1405, 1655] \\
     \textbf{Payments} & 2425 & 5215
    \end{tabular}
    \label{tab:tab15}
\end{table}

The adaptive marginal price payments are summarized in Table \ref{tab:tab16}. 

\begin{table}[ht]
    \centering
    \caption{Marginal pricing payments for each generator at each time period in the ARO multiperiod example.}
    \begin{tabular}{c|cc}
     \textbf{Generator} & 1 & 2 \\
     \hline
     $\mu_{t}^{*} u_{i, t}^{*}$ & [725, 725, 12675] & [225, 275, 4225] \\
     Uplifts & -11700  & 490 \\
     \textbf{Payments} & 2425 & 5215
    \end{tabular}
    \label{tab:tab16}
\end{table}

The non-adaptive dispatch in the ARO problem and the dispatch in the deterministic problem try to meet the same level of demand, namely [95, 100, 130]. However, the cost is higher in the ARO problem, because we protect against uncertainty in the load and in the capacity. For example, the commitments and dispatch are still feasible for a drop of [0, 7.5, 0.5] in the capacity of each generator and an increase of [10, 10, 2] in the load. So, the payments to the generators and the price $\boldsymbol{\mu}^{*}$ of electricity are higher compared to the deterministic case. In addition, the adaptive payment mechanisms we use do not feature a duality gap.

\section{Conclusion}
\label{sec:ada_conclusion}

In this work, we introduce the first pay-as-bid and marginal pricing methods for energy markets with non-convexities under uncertainty. We consider ARO formulations that protect against uncertainty in the load and capacity parameters and we provide the corresponding adaptive pricing schemes. We apply our method to realistic examples with increasing degrees of complexity and show, both theoretically and empirically, that it eliminates uplifts and corrections that are necessary in deterministic approaches.

\bibliographystyle{ieeetr}
\bibliography{ref}

\end{document}